\documentclass[conference, 10pt]{IEEEtran}
\ifCLASSINFOpdf
\else
\fi
\usepackage{amsfonts}
\usepackage{amsmath,amssymb}
\usepackage{cite}
\usepackage{graphicx}
\usepackage{hyperref}
\usepackage{wrapfig}
\usepackage{times}

\usepackage{graphicx}
\usepackage{color}
\usepackage{placeins}
\usepackage{float}
\usepackage{tabularx,colortbl}

\newtheorem{theorem}{Theorem}

\newtheorem{proposition}[theorem]{Proposition}
\newtheorem{remark}[theorem]{Remark}
\newtheorem{assumption}[theorem]{Assumption}

\IEEEoverridecommandlockouts                              
\begin{document}
%
\title{Robust Topology Identification and Control of LTI Networks}

\author{\IEEEauthorblockN{Mahyar Fazlyab and Victor M. Preciado}
\IEEEauthorblockA{Department of Electrical and Systems Engineering\\
University of Pennsylvania\\
Philadelphia, PA 19104-6228 USA\\
\texttt{\small \{mahyarfa,preciado\}@seas.upenn.edu}\\}
}


%


\maketitle

\begin{abstract}
This paper reports a robust scheme for topology identification and control of networks running on linear dynamics. In the proposed method, the unknown network is enforced to asymptotically follow a reference dynamics using the combination of Lyapunov based adaptive feedback input and sliding mode control. The adaptive part controls the dynamics by learning the network structure, while the sliding mode part rejects the input uncertainty. Simulation studies are presented in several scenarios (detection of link failure, tracking time varying topology, achieving dynamic synchronization) to give support to theoretical findings.
\end{abstract}


%
\IEEEpeerreviewmaketitle

\section{Introduction}
Complex networks are capable of modeling many real world dynamical processes which can be described by a set of interacting elements (nodes) having some sort of connections or causal relationships (edges) among them.
For example, in molecular biology, complex networks are used to describe regulatory relationships between transcription factors and their target genes \cite{lezon2006using}.  In a communication network, amount of traffic flowing between nodes constitute a dynamical network. \cite{kelly1998rate}

In recent years, there has been an extensive attention to topology identification and control of complex dynamical networks. Topology identification aims at finding the strength of connections between nodes, such as protein-DNA interactions in the regulation of various cellular processes. or detecting failures or anomalies in connections. Also, the ability to control the dynamic variables of nodes by external inputs is another interesting issue which has appreciable applications in, for example, synchronization of coupled oscillators \cite{Yu2009429}, or rate control in communication networks.

Topology identification of networks has been addressed in various research works. In \cite{yu2006estimating} for example, the topology of the network is estimated on-line using Lyapunov based adaptive feedback input. In \cite{5718112}, a node-knockout
procedure is proposed for the complete characterization of the interaction geometry in consensus-type networks. In a different approach addressed in \cite{shahrampour2013reconstruction} and \cite{shahrampour2013topology}, nodes are stimulated by wide-sense stationary input of unknown power spectral density and the topology of the network is identified via measuring cross power spectral densities of the outputs, which encode the direction and weights of the edges.  The control of complex networks has also been studied extensively. In \cite{li2006controlling}, for example, a linear state feedback controller is designed to synchronize the states with a desired orbit.

The purpose of this paper is two fold. We address both the control and topology identification of complex networks running on linear dynamics and in presence of input uncertainty. We assume that the topology of the underlying graph is unknown or uncertain. We design the input of the  network so as to enforce the dynamic variables of the nodes to track those of a predefined reference network, without having the knowledge of the network topology. In this setting, the unknown topology of the network is also estimated using Lypunov theory and adaptive feedback gains. Next, we develop the robust counterpart of the control-identification scheme using sliding mode technique. With this method the additive bounded disturbance to the input of the network is rejected, translating into smooth tracking and estimation. The rest of the paper is organized as follows: In section II the problem is defined and the tracking controller and the identifier are designed. In section III we robustify the algorithm through incorporating additional sliding mode control law. In section IV we provide simulations to verify the effectiveness of the proposed approach. Concluding remarks are drawn is section V.

\section{Simultaneous Identification and Control}
Consider the following linear time invariant dynamics on a directed network:

\begin{equation}
\dot{x}(t)=\mbox{A}x(t)+\mbox{B}u(t)  \label{1}
\end{equation}

where $\mbox{A}=[a_{ij}] \in \mathbb{R}^{N\times N}$ is the unknown weighted adjacency matrix of the directed network with $N$ nodes. $a_{ij}$ denotes the strength of connection from node $j$ to node $i$, and is zero when there is no connection. Positive or negative weights refer to excitatory or inhibitory effect. In consensus dynamics, the adjacency matrix $\mbox{A}$ is replaced by the negative Laplacian of the underlying network. Vector $x(t)\in \mathbb{R}^N$ captures the evolution of the nodes which can be, for instance, the transcription factor in gene regulatory networks or amount of traffic entering each node in a communication network. Here, without loss of generality, we are assuming scalar variables for each node. Extention to multi variable nodes, where each node is represented by an LTI system, is straightforward. Boolean matrix $\mbox{B}=[b_{ij}] \in \mathbb{R}^{n\times m}$, $b_{ij}\in \{0,1\}$ determines which nodes are accessible by external inputs, where $b_{ij}=1$  when input $j$ drives node $i$ and $b_{ij}=0$ otherwise. External input $u(t)\in \mathbb{R}^m$ is responsible for driving the states of the network. The goal is to control the dynamic variables of the nodes and to estimate the weights of connection between nodes. For this goal to be possible, we need structural controllability assumption of the pair $(\mbox{A},\mbox{B})$. In the context of networks, controllability translates into minimum number of driver nodes that can uniquely determine the time evolution of all nodes \cite{liu2011controllability}. 

\begin{remark}
If the topology of the network is completely unknown, all nodes must be accessible for input injection, i.e. $B=I_{N}$. However, if the boolean structure is known a priori and the weights of connections are unknown, fewer nodes are required for input injection. In the development of the identification algorithm, we consider the latter case.
\end{remark} 

To begin, we define the following reference network:

\begin{equation}
\dot{x}_m(t)=\mbox{A}_m x_m(t)+Br(t) \label{2}
\end{equation}

where $A_m$ is the reference adjacency matrix, and $r(t)$ is the reference control input. If the pair $(A_m,B)$ is controllable, one can generate the desired trajectory $x_m(t)$ using appropriate input $r(t)$.

\begin{assumption}
There exists matrices $K^* \in \mathbb{R}^{m\times n}$ such that the algebraic equations $A+BK^*=A_m$ is satisfied. For the special case of $B=I_{N}$, this assumption is held.
\end{assumption}

Consider the following input for the real network:
\begin{equation}
u(t)=K^*x(t)+L^*e(t)+r(t) \label{3}
\end{equation} 

Substituting Eq. \eqref{3} in Eq. \eqref{1} results in the following dynamics for the real network:

\begin{equation}
\dot{x}(t)=(A+BK^*)x(t)+Br(t)+BL^*e(t) \label{b}
\end{equation}

Subtracting \eqref{2} from \eqref{b} yields the error dynamics:

\begin{equation}
\dot{e}(t)=(A+BK^*)x(t)-A_mx_m(t)+BL^*e(t) \label{4}
\end{equation}

Now if $K^*$ and $L^*$ are deigned such that $A+BK^*=A_m$ and $A_m+BL^*$ is Hurwitz, the error dynamics reduces to $\dot{e}(t)=(A_m+BL^*)e(t)$, translating to asymptotic convergence of $e(t)$ to zero.

It is clear that when matrix $A$ is unknown or uncertain, the feedback matrix $K^*$ cannot be designed. By estimating $K^*$ and using the identity $A+BK^*$, one can recover the unknown topology. Denoting this estimate at time $t$ by $K(t)$, the closed loop network admits the following dynamics:
\begin{equation}
\dot{x}(t)=(A+BK(t))x(t)+Br(t)+BL^*e(t) \label{5}
\end{equation}

Similarly, by subtracting \eqref{2} from \eqref{5}, the error dynamics can be written as:
\begin{equation}
\dot{e}(t)=(A_m+BL^*)e(t)+B\tilde{K}(t)x(t) \label{6}
\end{equation}

where $\tilde{K}(t):= K(t)-K^*$ is the estimation residual.\\

\begin{proposition} 
If $K(t)$ satisfies the following differential equation, the tracking error will asymptotically converge to zero:
\begin{equation}
\dot{K}(t)=\dot{\tilde{K}}(t)=-W^{-1}B^TPe(t)x(t)^T, \label{7}
\end{equation}
where $P\in \mathbb{S}_{++}^{n\times n}$ and $W \in \mathbb{S}_{++}^{n\times n}$ are positive definite symmetric matrices.
\end{proposition}

\begin{proof}
We define the following Lyapunov function:
\begin{equation}
V=\frac{1}{2}e^TPe+\frac{1}{2}tr(\tilde{K}^TW\tilde{K}), \label{8}
\end{equation}
The first term is the weighted two norm of the tracking error, while the second terms is the weighted Forbenius norm of $\tilde{K}$ which vanishes only when $\tilde{K}$ is identically zero, resulting in a well defined Lyapunov function. Taking time derivative of Eq. \eqref{8} we see:
\begin{align}
\dot{V} =&-\frac{1}{2}e^TQe+x^T\tilde{K}^TB^TPe
+tr(\tilde{K}^TW\dot{\tilde{K}})\nonumber \\
\nonumber\\
=&-\frac{1}{2}e^TQe+tr(\tilde{K}^T(B^TPex^T+W\dot{\tilde{K}})) \label{9}
\end{align}
Here in the first equality, we have used the fact that $(A_m+BL^*)P+P(A_m+BL^*)^T=-Q$ for some positive definite matrix $Q$ for Hurwitz $A_m+BL^*$. In the second equality, we have used the identity $x^T\tilde{K}^TB^TPe=tr(\tilde{K}^TB^TPex^T)$. Substituting Eq. \eqref{7} in \eqref{9} yields:
\begin{equation}
\dot{V} =-\frac{1}{2}e^TQe \label{10}
\end{equation}
Let M be the set of all points for which $\dot{V}=0$, i.e. $M=\{(e,\tilde{K})|e=0\}$. According to Lasalle's invariance principle, all dynamic variables will asymptotically converge to the largest invariant M, where $e=0$, completing the proof. 
\end{proof}
The largest invariant set can be stated as:
\begin{equation}
M_{L}=\{( e,\tilde{K})|e=0 , \  \dot{\tilde{K}}=0 , \ B\tilde{K}x_m(t)=0 \}
\end{equation}  \label{M}
According to Eq. \ref{7}, the learning process is stopped. In order to have zero steady state estimation residual ($\tilde{K}$ in \eqref{M}), matrix $B$ must be full rank, and the reference signal $x_m(t)$ must span the N-dimensional space in time, i.e. it must satisfy the persistent excitation condition. Mathematically speaking, the condition of persistent excitation translates into the condition $\int\limits_{t}^{t+T} x_m(\tau)x_m^T(\tau)d\tau \geq \alpha I_N$ for some $\alpha , T>0$ and all $t \geq 0$ \cite{ioannou2012robust}. It can be readily verified that if $x_m(t)=[sin(\omega_0 t) \ sin(2\omega_0 t) ... sin(N\omega_0 t)]^T$ \textit{PE} condition is satisfied with $T=\frac{2\pi}{\omega_0}$ and $\alpha=\frac{\pi}{\omega_0}$.\\

\begin{remark}
In dynamic synchronization, i.e. when $x_1(t)=x_2(t)=...=x_N(t)$, the condition of persistent excitation is failed and the network topology becomes unidentifiable \cite{4806038}.\\
\end{remark}

\section{Robust Topology identification}
In practical settings, there usually exists uncertainty or disturbance to the input of the real network. In this case, the perturbation adversely affects tracking and learning. In order to reject this disturbance, we might add an additional control input to robustify the adaptation process. To see this, consider the following dynamics for the real network:
\begin{equation}
\dot{x}(t)=Ax(t)+B(u(t)+d(t))  \label{11}
\end{equation}
where $d(t)$ is the unknown additive disturbance with known bound. We resort to sliding mode technique which is a robust control scheme in order to reject this disturbance \cite{khalil2002nonlinear}. To begin, we first define the following sliding surface:
\begin{equation}
s(t)\triangleq\Gamma\left(e(t)-\int\limits_{0}^{t}(A_m+BL^*)e(\tau)d\tau\right) \label{12}
\end{equation}
Where $\Gamma\in\mathbb{R}^{m\times N}$ is a full rank matrix to be designed. By this definition, it is clear that once $s(t)=0,  \forall t\geq t_0$ for some $t_0\geq 0$ (or i.e. $s(t) \equiv \dot{s(t)}\equiv 0$), the tracking error would follow a stable manifold toward the origin. Hence, tracking becomes equivalent to regulation of $s(t)$ to zero. Consider the following input:
\begin{equation}
u(t)=K(t)x(t)+L^*e(t)+r(t)-\rho M\frac{s(t)}{\|s(t)\|} \label{13}
\end{equation}
Here $\|.\|$ refers to two norm. Substituting Eq. \eqref{13} in Eq. \eqref{11} results in the following dynamics for the real network:
\begin{align}
\dot{x}(t)=&\left(A+BK(t)\right)x(t)+B (r(t)+L^*e(t) \nonumber \\
&-\rho M\frac{s(t)}{\|s(t)\|}+ d(t)) \label{14}
\end{align}
Subtracting Eq.\eqref{2} from Eq. \eqref{14} results in the following error dynamics:
\begin{equation}
\dot{e}(t)=(A_m+BL^*)e(t)+B\left(\tilde{K}(t)x(t)-\rho M\frac{s(t)}{\|s(t)\|}+ d(t)\right) \label{15}
\end{equation}
\begin{proposition}
Consider the control input \eqref{13}. If $K(t)$ satisfies the differential equation \eqref{16}, and $\rho=\|P\Gamma B\|\|d(t)\|+\epsilon$ with $\epsilon >0$, the tracking error, as described by Eq. \ref{15} will asymptotically converge to zero.

\begin{equation}
\dot{K}(t)=\dot{\tilde{K}}(t)=-W^{-1}B^T\Gamma^TPs(t)x(t)^T \label{16} 
\end{equation}
\end{proposition}

\begin{proof}
Define the following Lypunov function candidate:
\begin{equation}
V=\frac{1}{2}s^TPs+\frac{1}{2}tr(\tilde{K}^TW\tilde{K}) \label{17}
\end{equation}
Differentiating Eq. \eqref{17} w.r.t. time yields:
\begin{align}
\dot{V}=& s^TP\Gamma\left[\dot{e}-(A_m+BL^*)e\right] \nonumber \\
=& s^TP\Gamma B(d-\rho M \frac{s}{\|s\|})
+s^TP\Gamma B\tilde{K}x + tr(\tilde{K}^TW\dot{\tilde{K}}) \nonumber \\
=& s^TP\Gamma B(d-\rho M \frac{s}{\|s\|})
+ tr(\tilde{K}^T(B^T\Gamma^TPsx^T+W\dot{\tilde{K}})) \nonumber \\
\label{18}
\end{align}
The differential equation \eqref{16} crosses out the second term of $\dot{V} $. Now by choosing $M=(P\Gamma B)^{-1}$ we get:
\begin{align}
\dot{V}=& s^TP\Gamma Bd -\rho\|s\| \nonumber \\
\leq& \|s\|(\|P\Gamma B\|\|d\|-\rho) \nonumber \\
\leq& -\epsilon \|s\| \label{19}
\end{align}
where the last inequality follows from the assumption $\rho=\|P\Gamma B\|\|d\|+\epsilon$. We maintain that inequality \eqref{19} ensures asymptotic convergence of the tracking error to zero as follows: Integrating the inequality \ref{19} implies $\epsilon \int\limits_{0}^{t}{\|s\|d\tau}\leq (V(0)-V(t))$. Since $\dot{V}$ is negative semi-definite and $V(t)$ is lower bounded, $\lim\limits_{t\to \infty}V(t)$ exists and hence by Schwartz inequality $\int\limits_{0}^{\infty}{\|s\|^2d\tau}\leq \left(\int\limits_{0}^{\infty}{\|s\|d\tau}\right)^2$, the function $g(t)\triangleq\frac{1}{2} \int\limits_{0}^{t}{\|s\|^2d\tau}$ has finite limit at $t=\infty$. Also $\ddot{g}(t)=s^T\dot{s}$ is bounded due to the boundedness of $s$ and $\dot{s}$. By Barbalat's lemma \cite{khalil2002nonlinear} one can conclude that $\lim\limits_{t\to\infty}s(t)=0$; and by the definition of $s(t)$, Eq. \eqref{12} convergence of $e(t)$ to zero is guaranteed.
\end{proof}

\section{NUMERICAL SIMULATIONS}
In this section, we present a simulation study to validate the effectiveness of the proposed approach. We consider a weighted undirected network whose weights are realizations of a uniform distribution in the unit interval, i.e. $a_{ij} \sim \mathcal{U}(0,1)$ for $i\neq j$ and $a_{ii}=0$.
Each node has its own input, i.e $B=I_{5\times 5}$. For the reference network, we assume $A_m=-I_{5\times5}$ and $B_m=I_{5\times 5}$. We also set the weights of the  Lyapunov function as $P=I_{5\times 5}$ and $W=10I_{5\times 5}$. The initial condition for the network is chosen as $x(0) \sim N(0,1)$. The initial condition for the reference network is zero, $x_m(0)=0$. The initial condition for the estimator $K(t)$ is also set to zero, $K(0)=0_{5\times 5}$. To fulfill the persistence excitation condition, the reference input is chosen as : $r_i(t)=r_{0i}sin(w_{i}t)$ where $r_{0i} \sim \mathcal{U}(1,2)$ and $w_{i} \sim \mathcal{U}(1,2)$
Figure.\ref{Fig1} illustrates the estimation of all elements of the adjacency matrix $A$.

\begin{figure}[h!]
  \centering
  \includegraphics[trim=1cm 0.5cm 1cm 1cm, clip=false, width=0.3\textwidth]{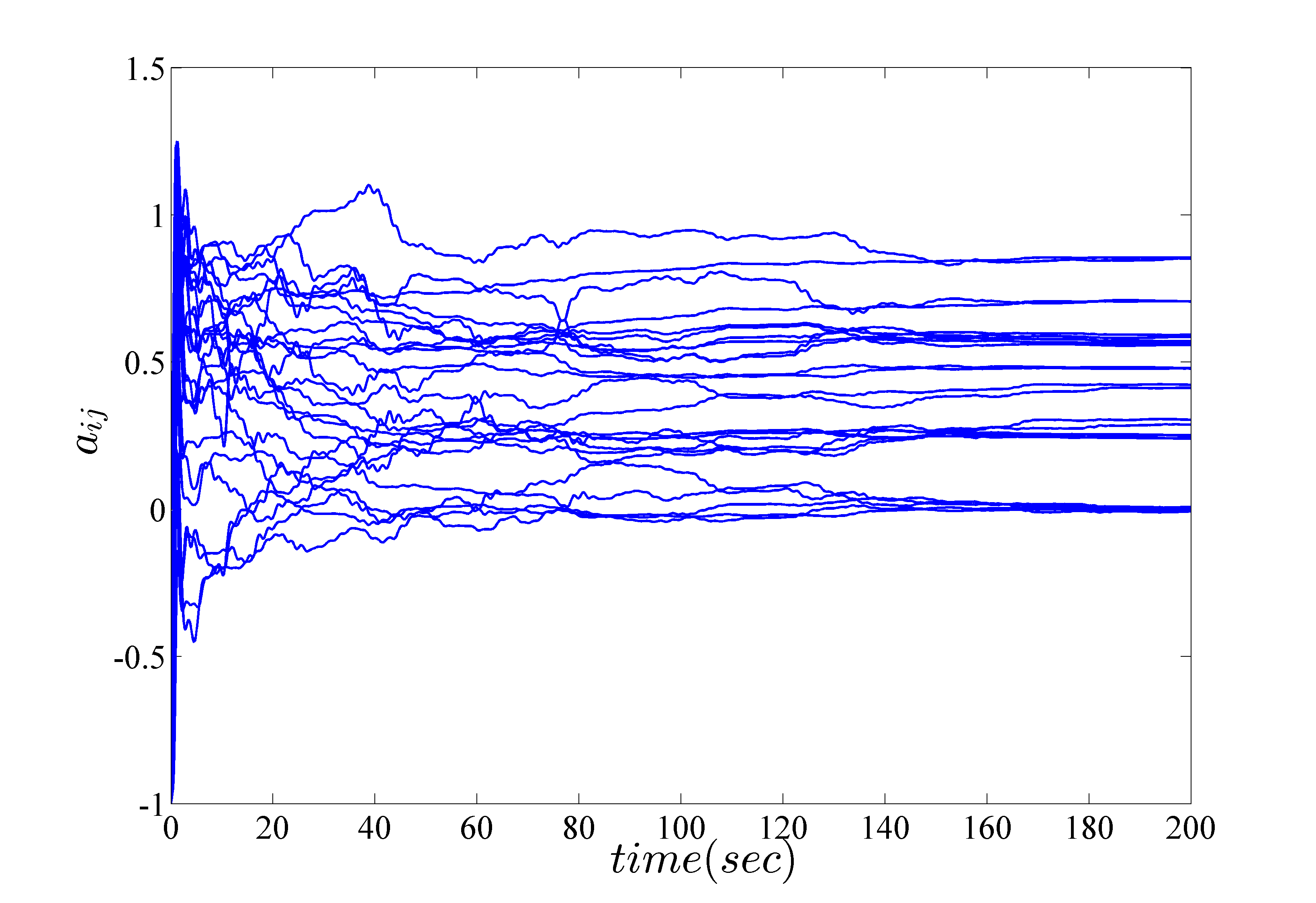}
  \caption{Estimation of weights}
  \label{Fig1}
\end{figure}

The proposed topology estimator can also be used for detecting link failures, i.e. sudden changes in the weights of the links as well as tracking slow temporal variations in the weights. Figure.\ref{Fig3} depicts the scenario when there is an abrupt failure in the link $a_{12}$. Figure.\ref{Fig4} shows the estimation behavior when the link $a_{12}$ has temporal variation $a_{12}(t)=0.56cos^2(2\pi t/800)$.  
\begin{figure}[h!]
  \centering
    \includegraphics[trim=1cm 1cm 1cm 1cm, clip=false, width=0.3\textwidth]{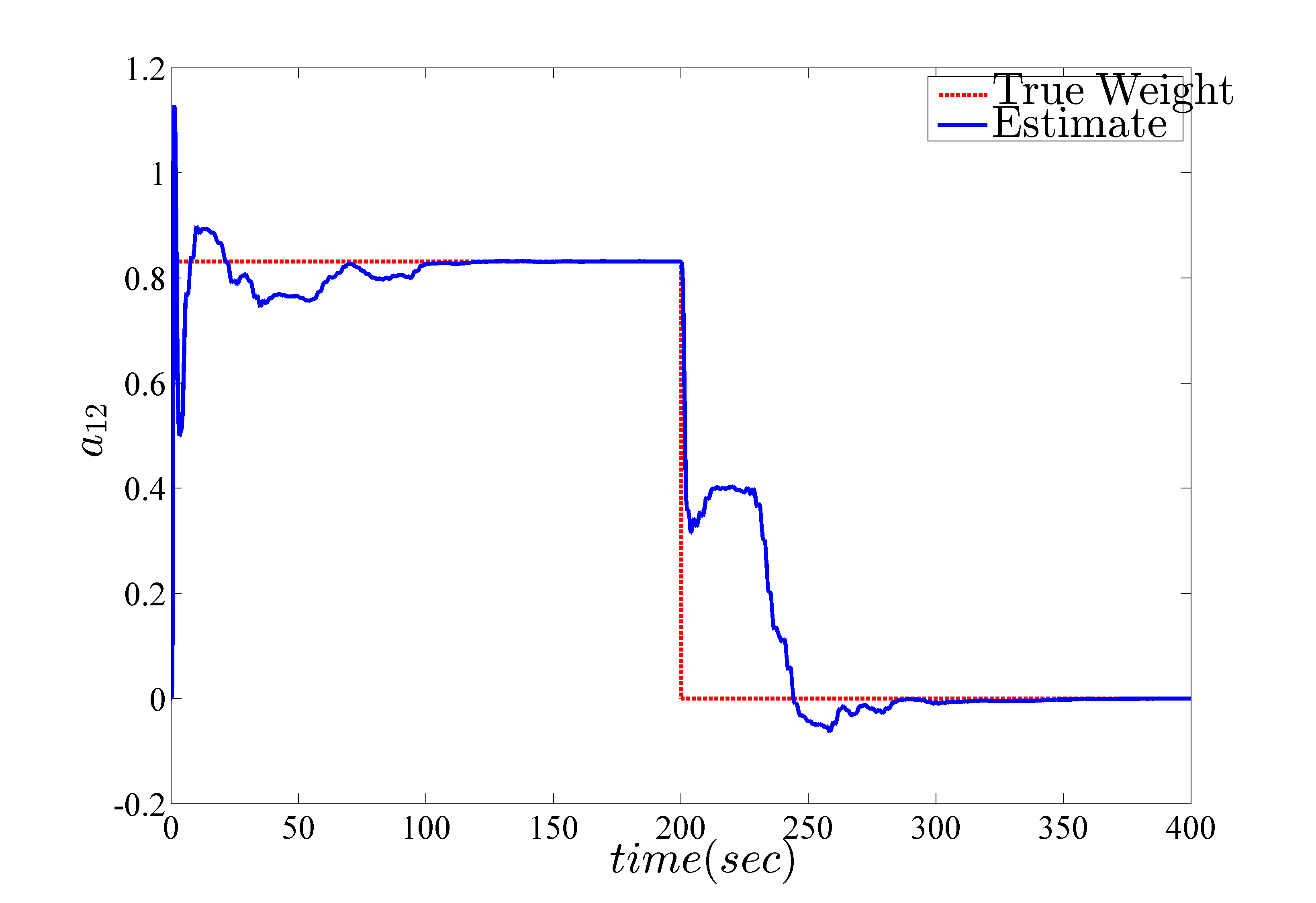}
  \caption{Detection of link failure between node 1 and node 2}
    \label{Fig3}
\end{figure}

\begin{figure}
  \centering
    \includegraphics[trim=1cm 1cm 1cm 1cm, clip=false,width=0.3\textwidth]{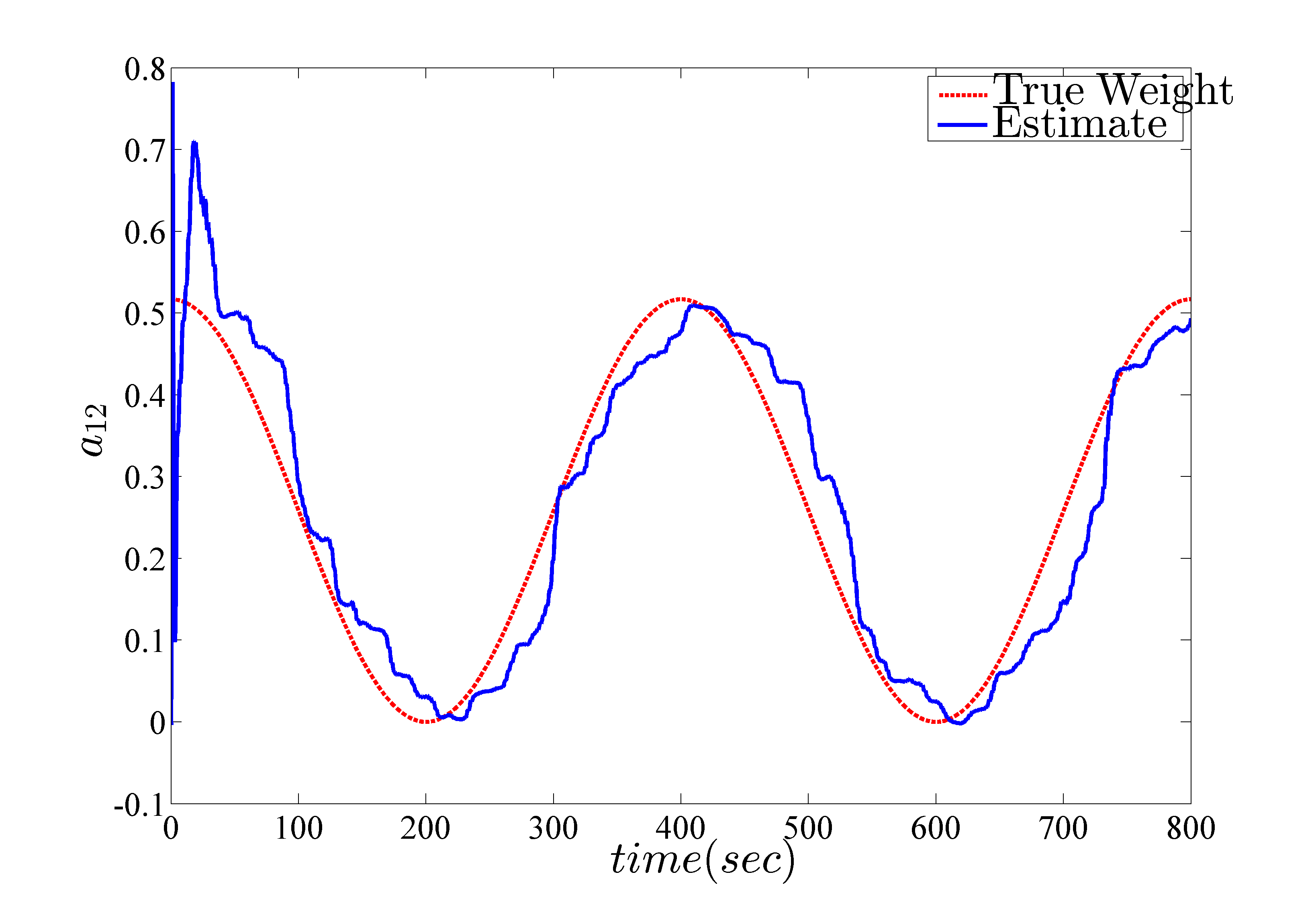}
  \caption{Tracking time varying connection strength between node 1 and node 2}
    \label{Fig4}
\end{figure}

In order to show the robustness of the identifier when sliding mode control is added, we consider a Gaussian disturbance added to all inputs, $d(t)\sim \mathcal{N}(0,1)$. Figure.\ref{Fig5} shows the sensitivity of learning in presence of noise. Figure.\ref{Fig6} verifies that the estimator perfectly smooths out the disturbance with the aid of sliding mode control. In figure.\ref{Fig7} we plot the state of node 1 along with the reference input and in presence of input disturbance. Note that for tracking purposes, the reference input may not be sufficiently rich, and hence the learning process will become useless.\\
\textit{Application in Dynamic Synchronization}: In order to show the effectiveness of the control algorithm on achieving dynamic synchronization, we choose the reference network as a dynamic consensus filter \cite{spanos2005dynamic} : $\dot{x}_m(t)=-Lx_m(t)+{r}(t)$ where $L$ is the Laplacian matrix of a connected undirected graph. Figure. \ref{Fig8} illustrates dynamic synchronization where $r(t)$ is the vector of unit step function and all nodes will asymptotically agree on unit ramp function.

\begin{figure}[h!]
  \centering
    \includegraphics[trim=1cm 1cm 1cm 1cm, clip=false,width=0.3\textwidth]{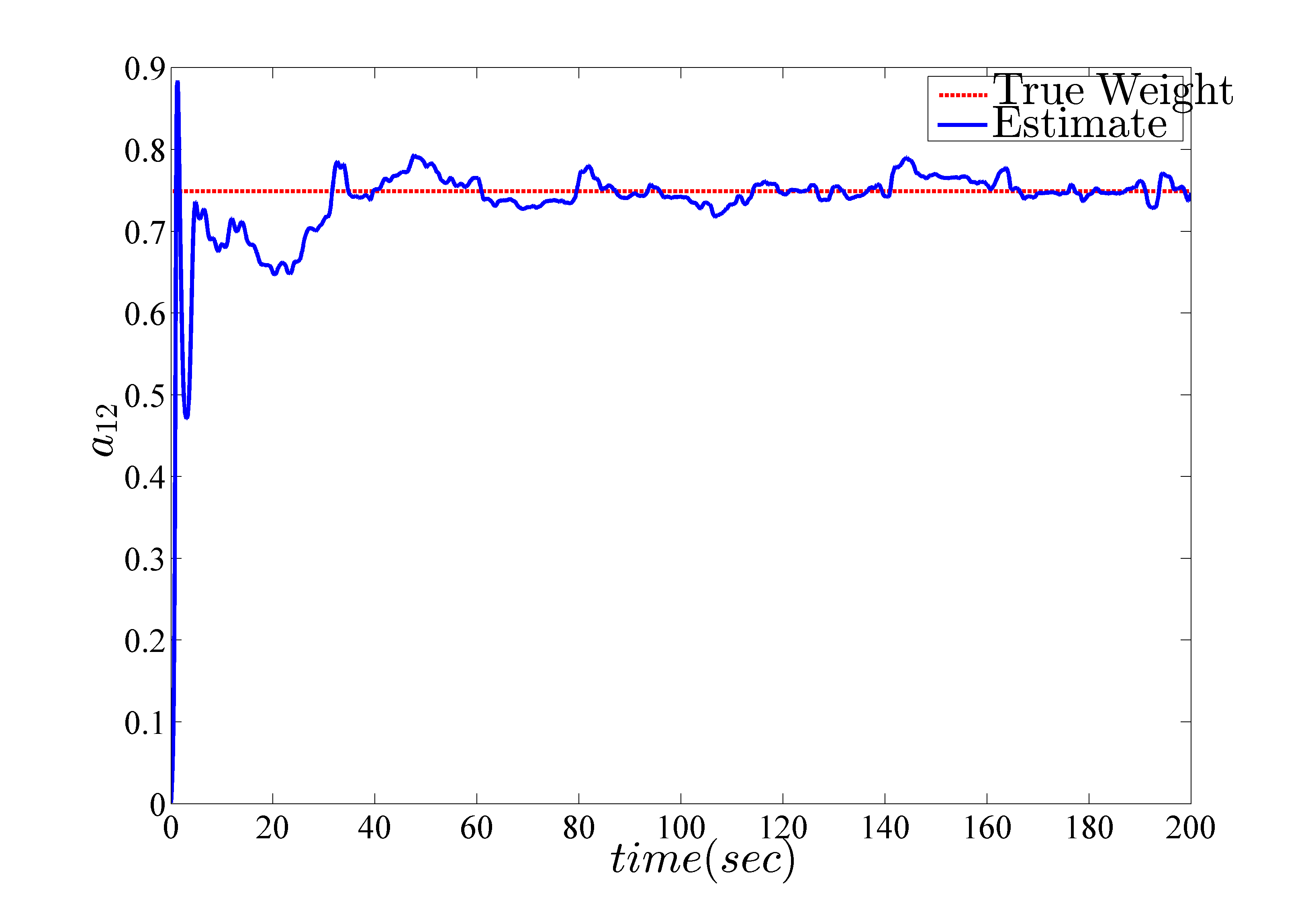}
  \caption{Estimation of link $a_{12}$ without disturbance rejection}
    \label{Fig5}
\end{figure}

\begin{figure}[h!]
  \centering
    \includegraphics[trim=1cm 1cm 1cm 1cm, clip=false,width=0.3\textwidth]{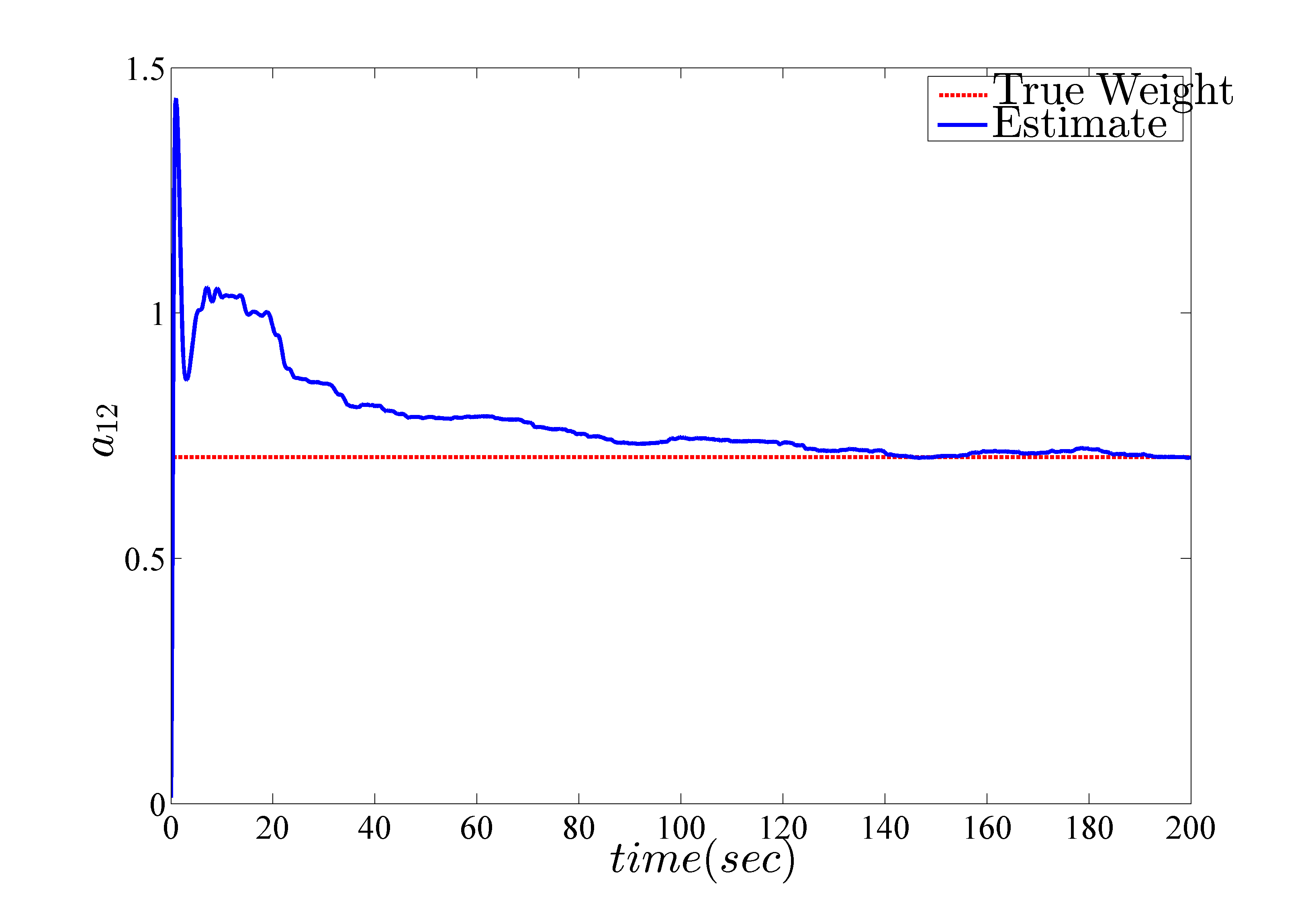}
  \caption{Estimation of link $a_{12}$ with disturbance rejection}
    \label{Fig6}
\end{figure}

\begin{figure}[h!]
  \centering
    \includegraphics[trim=1cm 1cm 1cm 1cm, clip=false,width=0.3\textwidth]{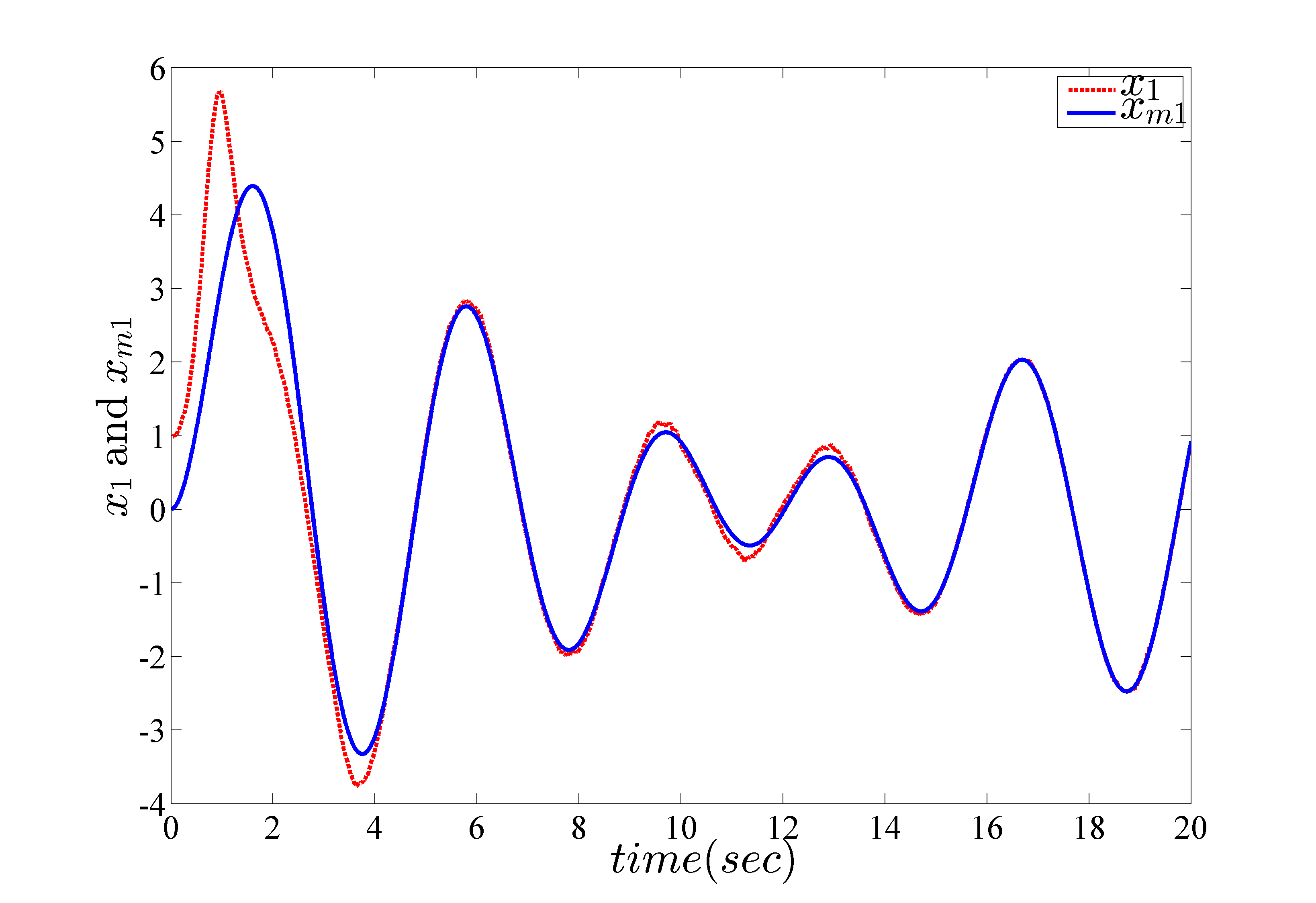}
  \caption{Tracking the reference signal in node 1 with disturbance rejection}
    \label{Fig7}
\end{figure}

\begin{figure}[h!]
  \centering
    \includegraphics[trim=1cm 1cm 1cm 1cm, clip=false,width=0.3\textwidth]{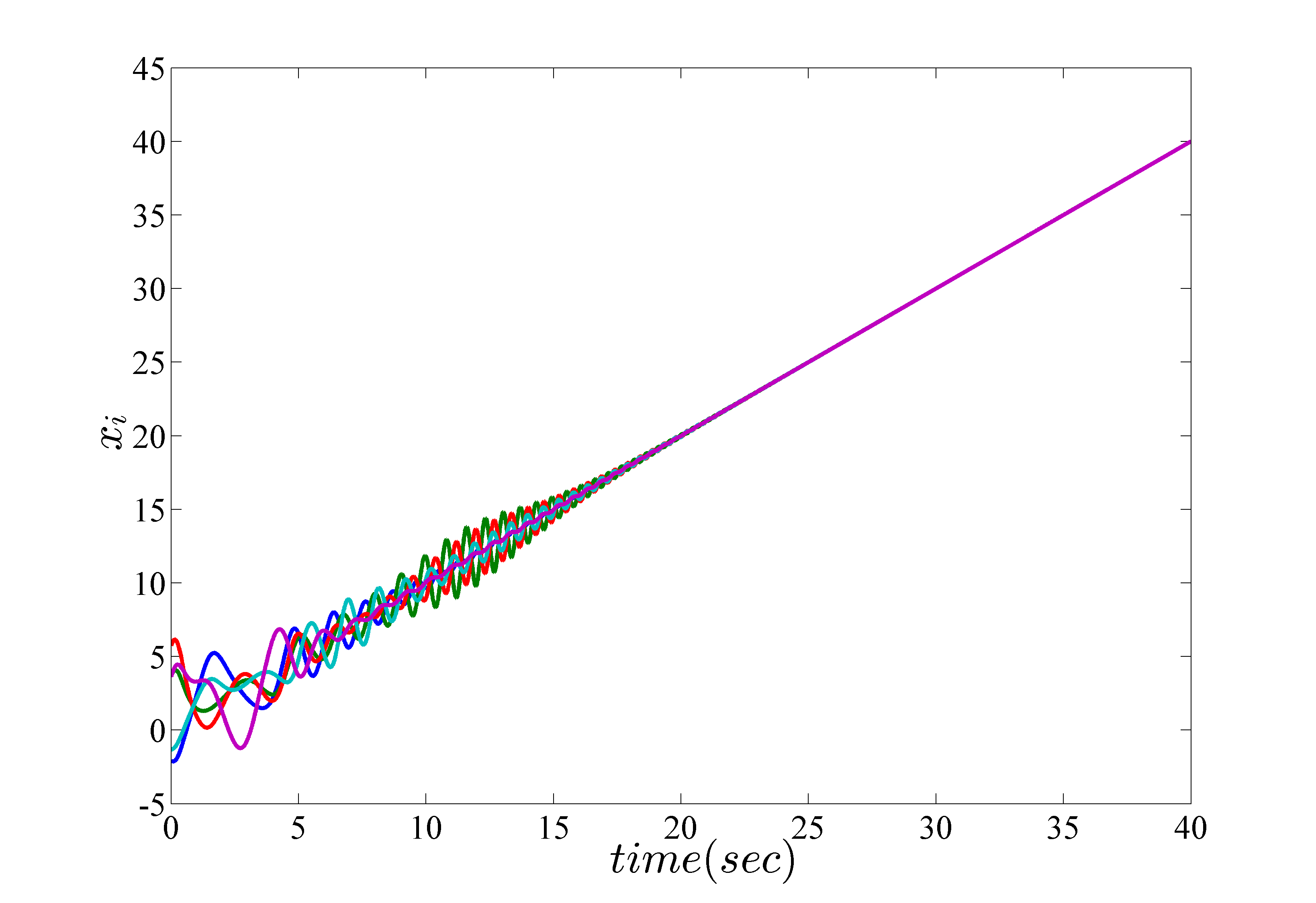}
  \caption{Dynamic Synchronization in the unknown network}
    \label{Fig8}
\end{figure}

\section{CONCLUSIONS}
This papers deals with topology identification and control of complex networks with linear dynamics. We used model reference adaptive scheme to design the proper input so as to drive the states of the nodes to desired trajectories. The proposed method is capable of controlling the nodes evolution without any information about the topology of the graph. The topology estimator is shown, via simulation, to be capable of tracking temporal variations in the weights of the connections. By combining the learning controller with sliding mode control, we could also robustify the algorithm to input uncertainty.

\addtolength{\textheight}{-12cm} 




\bibliographystyle{unsrt}
\bibliography{Refs}

\begin{thebibliography}{10}

\bibitem{lezon2006using}
Timothy~R Lezon, Jayanth~R Banavar, Marek Cieplak, Amos Maritan, and Nina~V
  Fedoroff.
\newblock Using the principle of entropy maximization to infer genetic
  interaction networks from gene expression patterns.
\newblock {\em Proceedings of the National Academy of Sciences},
  103(50):19033--19038, 2006.

\bibitem{kelly1998rate}
Frank~P Kelly, Aman~K Maulloo, and David~KH Tan.
\newblock Rate control for communication networks: shadow prices, proportional
  fairness and stability.
\newblock {\em Journal of the Operational Research society}, pages 237--252,
  1998.

\bibitem{Yu2009429}
Wenwu Yu, Guanrong Chen, and Jinhu Lü.
\newblock On pinning synchronization of complex dynamical networks.
\newblock {\em Automatica}, 45(2):429 -- 435, 2009.

\bibitem{yu2006estimating}
Dongchuan Yu, Marco Righero, and Ljupco Kocarev.
\newblock Estimating topology of networks.
\newblock {\em Physical Review Letters}, 97(18):188701, 2006.

\bibitem{5718112}
M.~Nabi-Abdolyousefi and M.~Mesbahi.
\newblock Network identification via node knock-out.
\newblock In {\em Decision and Control (CDC), 2010 49th IEEE Conference on},
  pages 2239--2244, Dec 2010.

\bibitem{shahrampour2013reconstruction}
Shahin Shahrampour and Victor~M Preciado.
\newblock Reconstruction of directed networks from consensus dynamics.
\newblock In {\em American Control Conference (ACC), 2013}, pages 1685--1690.
  IEEE, 2013.

\bibitem{shahrampour2013topology}
Shahin Shahrampour and Victor~M Preciado.
\newblock Topology identification of directed dynamical networks via power
  spectral analysis.
\newblock {\em arXiv preprint arXiv:1308.2248}, 2013.

\bibitem{li2006controlling}
Zhi Li, Gang Feng, and David Hill.
\newblock Controlling complex dynamical networks with coupling delays to a
  desired orbit.
\newblock {\em Physics Letters A}, 359(1):42--46, 2006.

\bibitem{liu2011controllability}
Yang-Yu Liu, Jean-Jacques Slotine, and Albert-L{\'a}szl{\'o} Barab{\'a}si.
\newblock Controllability of complex networks.
\newblock {\em Nature}, 473(7346):167--173, 2011.

\bibitem{ioannou2012robust}
Petros~A Ioannou and Jing Sun.
\newblock {\em Robust adaptive control}.
\newblock Courier Dover Publications, 2012.

\bibitem{4806038}
Liang Chen, Jun an~Lu, and C.K. Tse.
\newblock Synchronization: An obstacle to identification of network topology.
\newblock {\em Circuits and Systems II: Express Briefs, IEEE Transactions on},
  56(4):310--314, April 2009.

\bibitem{khalil2002nonlinear}
Hassan~K Khalil and JW~Grizzle.
\newblock {\em Nonlinear systems}, volume~3.
\newblock Prentice hall Upper Saddle River, 2002.

\bibitem{spanos2005dynamic}
Demetri~P Spanos, Reza Olfati-Saber, and Richard~M Murray.
\newblock Dynamic consensus on mobile networks.
\newblock In {\em IFAC world congress}. Prague Czech Republic, 2005.

\end{thebibliography}
\end{document}